\documentclass[11pt]{article}
\usepackage[english]{babel}
\usepackage[cp1252]{inputenc}
\usepackage{graphicx}
\usepackage{float}
\usepackage{amsmath, amssymb, amsthm}
\usepackage{indentfirst}
\newtheorem{Def}{Definition}[section]
\newtheorem{Prop}{Proposition}[section]
\newtheorem{Lema}{Lemma}[section]
\newtheorem{Teo}{Theorem}[section]
\newtheorem{Cor}{Corollary}[section]

\begin{document}

\title{\textsc{Homogeneous links and closed homogeneous braids}}

\author{Marithania Silvero \footnote{Partially supported by MTM2010-19355 and FEDER.}\\ \\
Departamento de Álgebra.
Facultad de Matemáticas. \\
Universidad de Sevilla.
Spain.\\
{\tt marithania@us.es}\\ \\
}

\maketitle


\noindent \textbf{Abstract} \,
Is any positive knot the closure of a positive braid? No. But if we consider positivity in terms of the generators of the braid group due to Birman, Ko and Lee, then the answer is yes. In this paper we prove that the same occurs when considering homogeneity.

\vspace{0.1cm}

In the way we prove that the plumbing of two surfaces is a BKL-homogeneous surface if and only if both summands are BKL-homogeneous surfaces, a parallel result to that of Rudolph involving quasipositive surfaces. BKL-homogeneous surfaces are, in fact, a natural generalization of quasipositive surfaces.

\bigskip

\noindent \textbf{Keywords:} Homogeneous links. Pseudoalternating links. Braids. Seifert surface. Stallings plumbing.
\vspace{1.2cm}
\mbox{ }


\section{Introduction}

The notion of ``homogeneity'' exists for links and also for braids. In the latter case this notion is defined with respect to a given generating set, and depends on this choice. In this paper we study the connection between these notions.

\vspace{0.2cm}

Homogeneous links were introduced by Peter Cromwell in \cite{CromwellHom}; a link is homogeneous if it has a homogeneous diagram, that is, a diagram in which all the edges of each block of its Seifert graph have the same sign (see Section 2).

\vspace{0.2cm}

The braid group on $n$ strands, $B_n$, has a standard presentation due to Artin (\cite{Artin1}, \cite{Artin2}) with $\sigma_1, \ldots, \sigma_{n-1}$ as generators. A braid is said to be Artin-homogeneous if it can be represented by a homogeneous standard braid word, that is, a braid word where each Artin generator appears always with the same sign. Note that usually these braids are simply called homogeneous braids \cite{Stallings}; as we are going to work with two different presentations of $B_n$, we will refer to these braids as Artin-homogeneous.

\vspace{0.2cm}

It is obvious that the closure of any Artin-homogeneous braid is a homogeneous link. However, not every homogeneous link is the closure of an Artin-homogeneous braid, as we prove in Proposition \ref{homnobraidhom}.

\vspace{0.2cm}

J. S. Birman, K.H. Ko and S.J. Lee gave in \cite{BirmanKoLee} a new presentation of the braid group $B_n$ on $n$ strands; in this presentation, generators are given by $\sigma_{rs}$, with $1\leq r < s\leq n$, where $\sigma_{rs} = (\sigma_{s-2} \ldots \sigma_{r})^{-1} \sigma_{s-1} (\sigma_{s-2} \ldots \sigma_{r})$. Define BKL-homogeneous braids as those braids which can be expressed by a homogeneous BKL-word, that is, a braid word using Birman-Ko-Lee generators with each generator appearing always with the same sign. A link is BKL-homogeneous if it is the closure of a BKL-homogeneous braid.

\vspace{0.2cm}

In this paper we prove that every homogeneous link is BKL-homogeneous (Corollary \ref{corolariohom1}). The converse is not true: we will see that the non-homogeneous knot $9_{48}$ is the closure of a BKL-homogeneous braid.

\vspace{0.2cm}

To show that homogeneous implies BKL-homogeneous, we extend a result by Lee Rudolph \cite{Rudolph1} to the case of BKL-homogeneous surfaces. More precisely, BKL-homogeneous braided surfaces are Seifert surfaces whose boundaries are BKL-homogeneous links. A surface is BKL-homogeneous if it is ambient isotopic to a BKL-homogeneous braided surface. The main result of this paper is the following:

\begin{Teo}\label{Teorema}
Let $S = S_1 * S_2$ be a Stallings plumbing of Seifert surfaces $S_1$ and $S_2$; then $S$ is a BKL-homogeneous surface if and only if both $S_1$ and $S_2$ are BKL-homogeneous surfaces.
\end{Teo}

%

As a consequence of Theorem \ref{Teorema} we prove in Corollary \ref{corolariopseudoalt} that pseudoalternating links are BKL-homogeneous, where pseudoalternating links are a family of links introduced by E.J. Mayland and K. Murasugi in \cite{Pseudoalternantes} which contains the family of homogeneous links.

\vspace{0.2cm}

The plan of the paper is as follows. In Section 2 we recall the definitions of homogeneous and pseudoalternating links. Section 3 is devoted to set some definitions involving braids: homogeneous braids and BKL-homogeneous braids; we also show the relation between homogeneous links and homogeneous braids. In Sections 4 we recall concepts such as star and patch, which will be used in Sections 5 and 6 in order to define braided surfaces and a special type of Stallings plumbing involving them. In Section 7 we present the main Theorem, and prove that every homogeneous link is BKL-homogeneous.

\section{Homogeneous and pseudoalternating links}

Given an oriented diagram $D$ of a link $L$, the surface $S_D$ obtained by applying Seifert's algorithm \cite{LibroCromwell} is known as projection surface of $L$ associated to $D$.

\vspace{0.2cm}

Given a projection surface $S_D$, we can construct a graph $G_D$ as follows: associate a vertex to each Seifert disc and draw an edge connecting two vertices if and only if their associated Seifert discs are connected by a band; each edge must be labelled with the sign $+$ or $-$ of its associated crossing in $D$. The graph $G_D$ is called the Seifert graph associated to $D$.

\vspace{0.2cm}

\begin{figure}[t]
\centering
\includegraphics[width = 13cm]{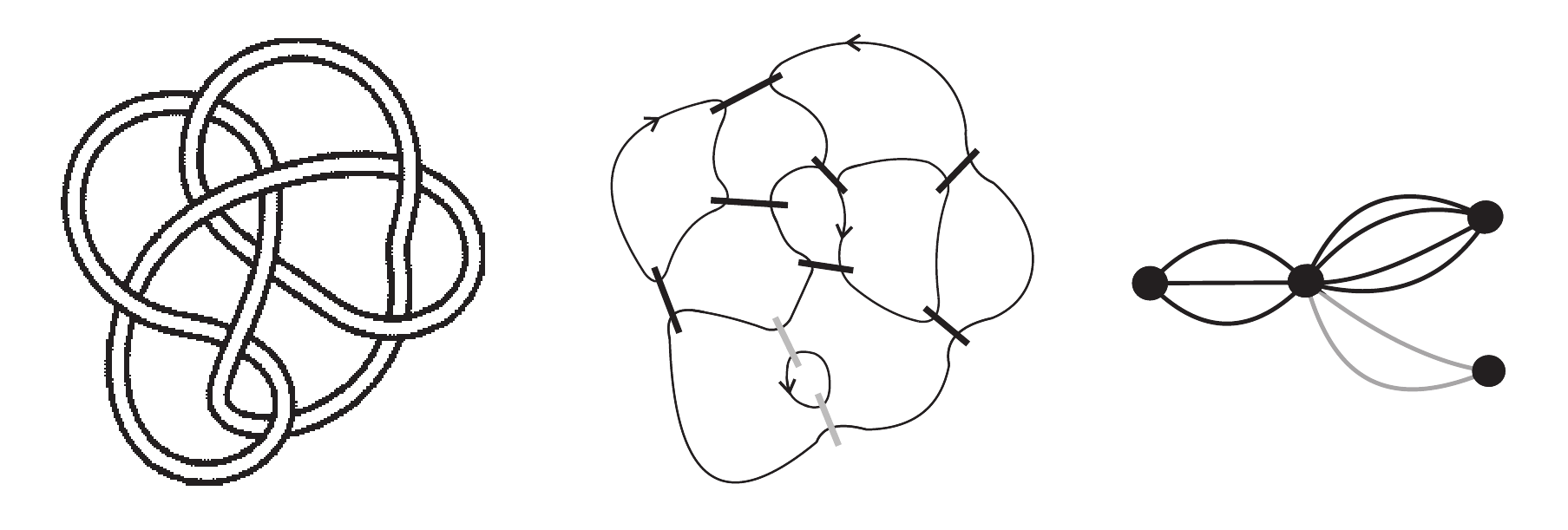}
\caption{\small{A diagram of the knot $9_{43}$ (taken from \cite{knotatlas}), the diagram obtained when applying Seifert's algorithm and its associated Seifert graph. Dark and light colors represent positive and negative signs, respectively. $9_{43}$ is homogeneous, as the diagram in the figure is so.}}
\label{943}
\end{figure}

Given a connected graph $G$, a vertex $v$ is a cut vertex if $G \backslash \{v\}$ is disconnected. Cutting $G$ at all its cut vertices produces a set of connected subgraphs containing no cut vertices, each of which is called a block of the graph. A Seifert graph is homogeneous if all the edges of a block have the same sign, for all blocks in the graph.

\begin{Def}
An oriented diagram $D$ is homogeneous if its associated Seifert graph $G_{D}$ is homogeneous. A link is homogeneous if it has a homogeneous diagram.
\end{Def}

Homogeneous links were introduced by Peter Cromwell in \cite{CromwellHom}; a more general class is the class of pseudoalternating links, introduced by E.J. Mayland and K. Murasugi in \cite{Pseudoalternantes}, which we define next.

\vspace{0.2cm}

Primitive flat surfaces are those projection surfaces arising from positive or negative diagrams that has not nested Seifert circles. A generalized flat surface is an orientable surface obtained by a finite iteration of Stallings plumbings, using primitive flat surfaces as the bricks of the construction (see Figure \ref{dibujopegadopseudoalt}). The Stallings plumbing or Murasugi sum will be defined in Section 6.

%
%

\begin{Def}
A link is said to be pseudoalternating if it is the boundary of a generalized flat surface.
\end{Def}

Since the projection surface constructed from any homogeneous diagram is a generalized flat surface, it follows that every homogeneous link is pseudoalternating. The converse is a consequence of a conjecture by Kauffman \cite{LibroKauffman}; Mayland and Murasugi posed a similar question in \cite{Pseudoalternantes}.

\vspace{0.2cm}

\section{Artin-homogeneous and BKL-homogeneous braids}

A braid can be represented by a braid word, using the generators of the standard presentation of the $n$ strands braid group $B_n$ given by Artin in \cite{Artin1}, \cite{Artin2}; the generator $\sigma_i$ represents a crossing involving strands in positions $i$ and $i+1$.
$$
B_n = \left< \sigma_1, \sigma_2, ... , \sigma_{n-1} \left| \begin{array}{cccc}
                                                            \sigma_i\sigma_j\sigma_i = \sigma_j\sigma_i\sigma_j & & &  |i-j| = 1 \\
                                                            \sigma_i\sigma_j = \sigma_j\sigma_i & & &  |i-j| > 1
                                                            \end{array}
                                                    \right.
\right>
$$

\begin{figure}[t] \label{dibujopegadopseudoalt}
\centering
\includegraphics[width = 12cm]{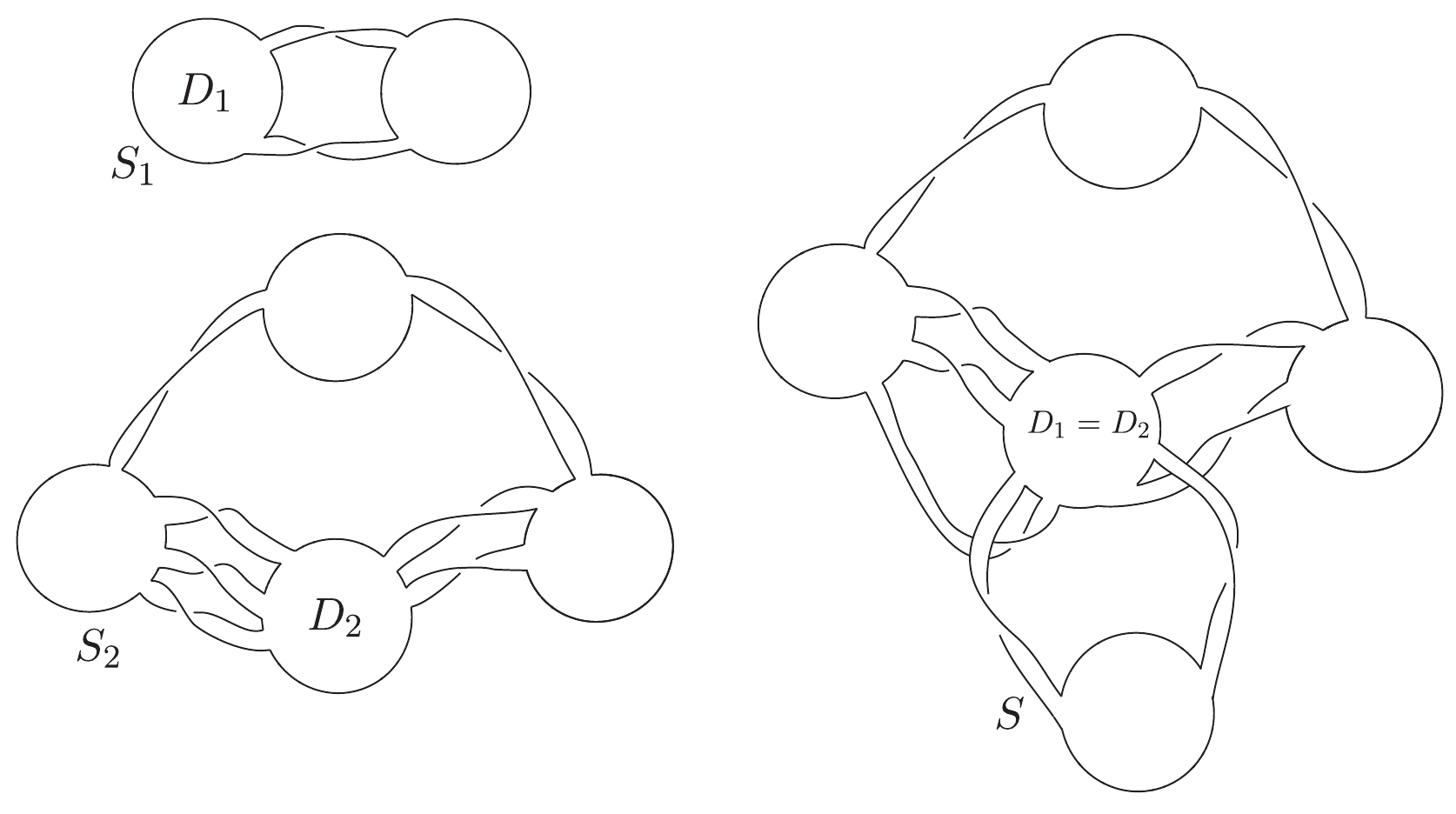}
\caption{\small{$S_1$ and $S_2$ are primitive flat Seifert surfaces. By an identification of discs $D_1$ and $D_2$, the generalized flat surface $S$ is obtained. The link spanned by $S$ is a pseudoalternating link.}}
\label{pseudoalt}
\end{figure}

\begin{Def}\label{artinbraids}
A braid word is said to be homogeneous if for each $i$, the exponents of all occurrences of $\sigma_i$ have all the same sign. A braid is homogeneous if it can be represented by a homogeneous word.
\end{Def}

\vspace{0.2cm}

Since we are going to work with two different presentations of the braid group $B_n$ in this paper, braid words written with the classical Artin generators will be called Artin-words. Thus, homogeneous words and homogeneous braids in Definition \ref{artinbraids} become homogeneous Artin-words and Artin-homogeneous braids respectively. Closures of Artin-homogeneous braids are Artin-homogeneous links.

\vspace{0.2cm}

Peter Cromwell stated in \cite{CromwellHom} that there are homogeneous links which cannot be presented as Artin-homogeneous braids, without giving a proof; here, we give a proof of this result by finding an example of these links.

\begin{Prop}\label{homnobraidhom}
There are homogeneous links which are not closure of any Artin-homogeneous braid.
\end{Prop}

\begin{proof} The knot $5_2$ is positive, hence homogeneous; it is for example the closure of the braid $\beta = \sigma_2^{-3}\sigma_1^{-1}\sigma_2\sigma_1^{-1}$, which is not a homogeneous Artin-word. Suppose that $5_2$ is the closure of a braid $\gamma$ represented by a homogeneous Artin-word $w$. We take $w$ of minimal length among all homogeneous Artin-words whose associated closed braid is $5_2$. Let $D$ be the associated homogeneous diagram. As projection surfaces constructed from homogeneous diagrams have minimal genus \cite{CromwellHom}, $g(S_D) = g(5_2) =1$ leads to $s + 1 = c$, where $s$ and $c$ are the number of Seifert discs and bands in $S_D$. Notice that $s$ is the number of strands and $c$ is the number of crossings of $\gamma$.

Since $c \geq 5$, $\gamma$ must have at least 4 strands, and then some generator $\sigma_i$ must appear at most once. All generators must appear, since $5_2$ is a knot (a one-component link), so there exists one generator appearing exactly once, and this is a nugatory crossing. This is a contradiction with the minimality of $w$ since $5_2$ is prime.
\end{proof}

\vspace{0.2cm}

The braid group $B_n$ admits another well known presentation due to Birman, Ko and Lee \cite{CromwellHom}; $\sigma_{ij}$ means strands $i$ and $j$ cross passing in front of the other strands. That is, $\sigma_{ij}= (\sigma_{j-2} \ldots \sigma_i)^{-1} \sigma_{j-1} (\sigma_{j-2} \ldots \sigma_{i})  $, with $i < j$:

$$
B_n = \left< \sigma_{rs}, \, 1 \leq r < s \leq n \left| \begin{array}{cccc}
                                                            \sigma_{st}\sigma_{qr} = \sigma_{qr}\sigma_{st} & & & (t-r)(t-q)(s-r)(s-q) > 0 \\
                                                            \sigma_{st}\sigma_{rs} = \sigma_{rt}\sigma_{st} = \sigma_{rs}\sigma_{rt} & & & 1 \leq r < s < t \leq n
                                                            \end{array}
                                                    \right.
\right>
$$

A braid word in terms of Birman-Ko-Lee generators will be called a BKL-word. In a positive (negative) BKL-word only positive (negative) exponents occur. A homogeneous BKL-word is a BKL-word in which for each $i, j$, the exponents of all occurrences of $\sigma_{ij}$ have all the same sign.

\begin{figure}[t]
\centering
\includegraphics[width = 9cm]{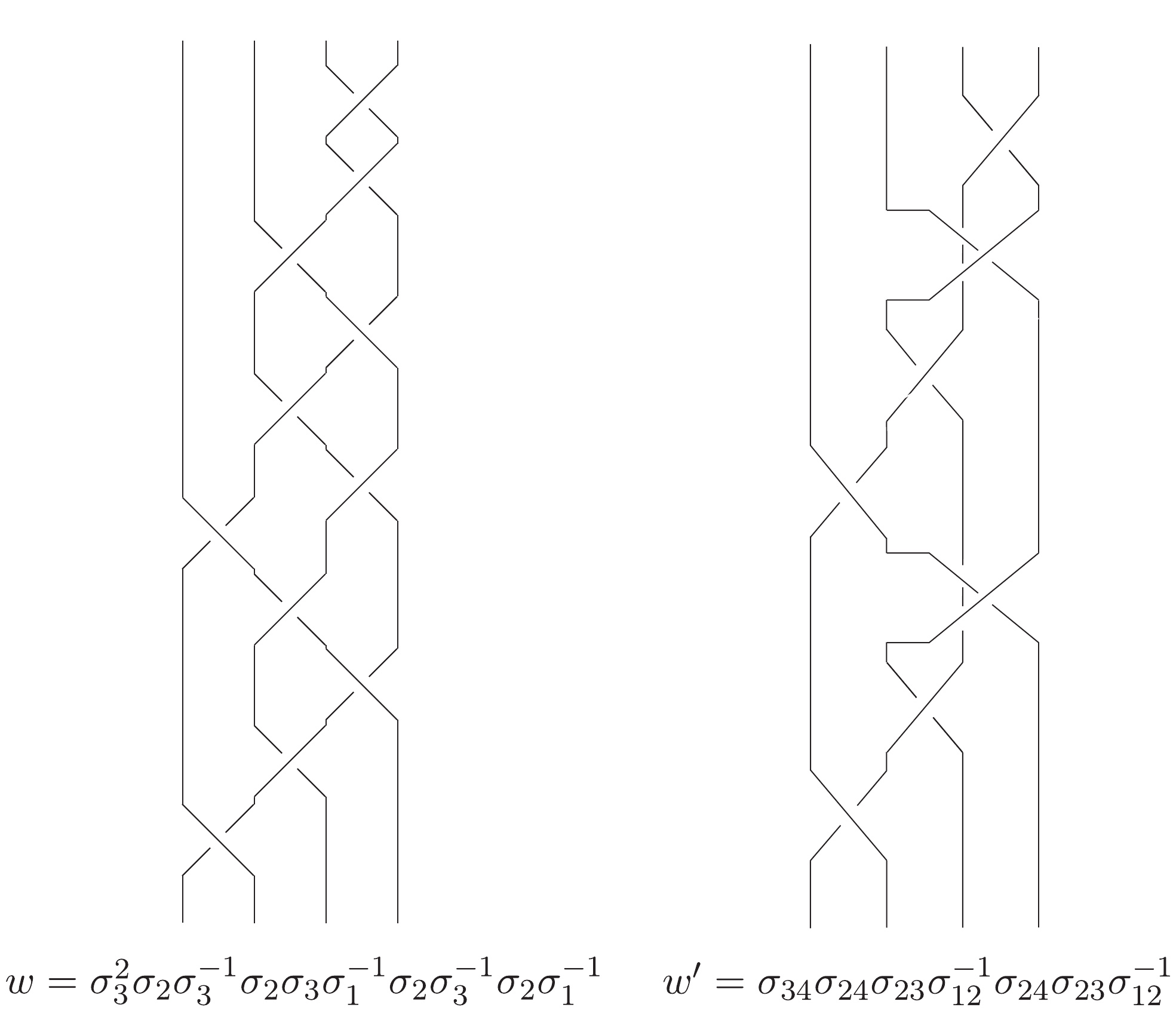}
\caption{\small{$w$ and $w'$ are two equivalent words expressing a same braid, $\beta$. As $w'$ is a homogeneous BKL-word, $\beta$ is a BKL-homogeneous braid. Knot $9_{48}$ can be expressed as the closure of $\beta$, thus $9_{48}$ is a BKL-homogeneous link.}}
\label{948}
\end{figure}

\begin{Def}
A BKL-positive (negative) braid is a braid which can be expressed by a positive (negative) BKL-word. If a link is the closure of a BKL-positive (negative) braid, we call it BKL-positive (negative) link.
\end{Def}

Lee Rudolph proved in \cite{Rudolph2} that positive links are BKL-positive links. Taking the mirror image of the link, it follows that negative links are BKL-negative links. The converse is not true, however a theorem by Baader \cite{Baader} states that a knot is positive if and only if it is homogeneous and BKL-positive.

\begin{Def}
A braid is BKL-homogeneous if it can be represented by a homogeneous BKL-word. The closure of a BKL-homogeneous braid is a BKL-homogeneous link.
\end{Def}

Note that BKL-positive and BKL-negative links are BKL-homogeneous. We will prove that every homogeneous link is a BKL-homogeneous link. The converse is not true:

\begin{Prop}\label{BKLnobraidhom}
There are BKL-homogeneous links which are not homogeneous.
\end{Prop}

\begin{proof} On one hand the knot $9_{48}$ (Figure \ref{948}) is not homogeneous \cite{CromwellHom}, and on the other hand it is the closure of the braid $\gamma = [w]$, where $w = \sigma_3^2\sigma_2\sigma_3^{-1}\sigma_2\sigma_3\sigma_1^{-1}\sigma_2\sigma_3^{-1}\sigma_2\sigma_1^{-1}$; since $\sigma_{i i+2} = \sigma_{i+1}\sigma_i\sigma_{i+1}^{-1}$, then $\gamma = [w']$, where $w' = \sigma_{34}\sigma_{24}\sigma_{23}\sigma_{12}^{-1}\sigma_{24}\sigma_{23}\sigma_{12}^{-1}$ is a homogeneous BKL-word.
\end{proof}

\section{Stars and patches}\label{seccionestrellas}

Throughout this section, every surface will be assumed to be compact and oriented. A Seifert surface is in addition connected and its boundary is non-empty. In a handle decomposition of a surface $S$

$$
S= \left( \bigcup_{x \in X} d_x \right) \, \, \bigcup \, \,  \left( \bigcup_{z \in Z} b_z \right),
$$

\noindent we refer to the (disjoint) 0-handles $d_x$ as discs, and to the (disjoint) 1-handles $b_z$ as bands. Write $S^d = \bigcup_{x \in X} d_x$ and $S^b = \bigcup_{z \in Z} b_z$. The attaching regions are the connected components of $\partial S^d \bigcap \partial S^b = S^d \bigcap S^b$.

\vspace{0.2cm}

Let us define $\#_d (S) := \mbox{card } (X)$ and $\#_b (S) := \mbox{card } (Z)$; in other words, given a surface decomposed as above $\#_d$ counts its number of discs and $\#_b$ its number of bands.

\vspace{0.2cm}

Given a point $p \in S$, write $d_p$ or $b_p$ for the unique disc or band containing $p$; $p$ is in an attaching region if and only if $p$ is contained in a disc and a band.

\begin{Def}
Let $S$ be a surface. An n-star $\varphi \subset S$ consists of an interior point $c$ of $S$, called the center of $\varphi$, together with $n$ arcs (called rays) $\tau_1, \ldots, \tau_n$ pairwise disjoint except at $c$, each one going from $c$, along the interior of $S$, to a point in $\partial S$. The final point of $\tau_i$ will be denoted $tip(\tau_i)$.

A regular neighborhood $N_s(\varphi)$ of an n-star $\varphi \subset S$ is called an n-patch.
\end{Def}

An n-patch can be thought as a polygon with 2n edges, which are alternately boundary arcs and proper arcs in the surface.

\vspace{0.2cm}

Let $S$ be a surface with a handle decomposition as above; we say that an n-star $\varphi \subset S$ is transverse to the decomposition of $S$ if the center of the n-star lies in $\mbox{Int } S^d$ and each ray $\tau_i$ is transverse to $S^d \bigcap S^b$, with $tip(\tau_i) \in \partial S^d \backslash S^b$ (that is, $tip(\tau_i)$ belongs to the boundary of a disc, but not to the attaching regions).

\vspace{0.2cm}

Let $\varphi$ be a transverse n-star; from now on, ``arc'' will mean a connected component of either $\varphi \bigcap S^d$ or $\varphi \bigcap S^b$, that is, the arcs obtained by cutting the n-star using the attaching regions as blades (the arc containing the center of the star is not properly an arc); for each ray $\tau \subset \varphi$, write $\delta_b(\tau)$ for the number of arcs contained in $\tau$ which lie on bands of the surface, that is, $\delta_b(\tau)$ is the number of connected components of $\tau \bigcap S^b$. Then $\delta_b(\varphi) : = \displaystyle \sum_{\tau \subset \varphi} \delta_b(\tau)$ counts the number of times that the n-star crosses the bands. If $\delta_b(\tau) > 0$ we say that $\tau$ is a long ray.
An n-star $\varphi \subset S$ is minimal with respect to the decomposition of the surface if $\delta_b(\varphi) \leq \delta_b(\varphi')$ for every n-star $\varphi' \subset S$ transverse to the decomposition and ambient isotopic to $\varphi$ in $S$.

\vspace{0.2cm}

Let $\tau \subset \varphi$ be a long ray; we denote $tail(\tau)$ the unique arc containing $tip(\tau)$, which lies on a disc, $d_{tip(\tau)}$. The other endpoint of the tail is the coccyx of $\tau$, which is on an attaching region. The arc $tail(\tau)$ divides $d_{tip(\tau)}$ in two discs; if at least one of them is disjoint with the bands in $S$, except for $b_{coccyx(\tau)}$, the ray $\tau$ is called loose. We say that $\tau$ is slack if it contains an arc with both endpoints on the same attaching region.

\vspace{0.2cm}

\begin{Lema}[Rudolph \cite{Rudolph1}]
If $\tau$ is a ray of a minimal n-star, then $\tau$ is neither slack nor loose.
\end{Lema}

\begin{proof}
Let $\tau$ be a slack ray; it has a slack arc related to an attaching region. The innermost slack arc associated to this attaching region (that is, the one with no arcs in the disc bounded by itself and the segment joining its endpoints), can be removed by pushing it to the other side of the attaching region. The resulting n-star $\varphi'$ is ambient isotopic to the original one, but $\delta_b(\varphi') < \delta_b(\varphi)$. Hence $\varphi$ is not minimal; a contradiction.

Suppose $\tau$ is not slack. If $\tau$ were loose, there could exist other rays whose coccyx and tips are in the segment joining $coccyx(\tau)$ and $tip(\tau)$; as $\varphi$ cannot intersect itself, these rays must have their coccyx in the same attaching region as $\tau$, so they are loose too. The innermost one, $\tau'$, can be removed, by performing an ambient isotopy pulling $tip(\tau')$ along the boundary of the surface, crossing back the last band the ray had crossed, $b_{coccyx(\tau')} = b_{coccyx(\tau)}$, and obtaining an n-star $\varphi'$ with $\delta_b(\varphi') = \delta_b(\varphi)-1$, yielding again a contradiction.
\end{proof}

\section{Braided surfaces}\label{seccionbraidedsurfaces}

\begin{figure}[t]
\centering
\includegraphics[width = 14cm]{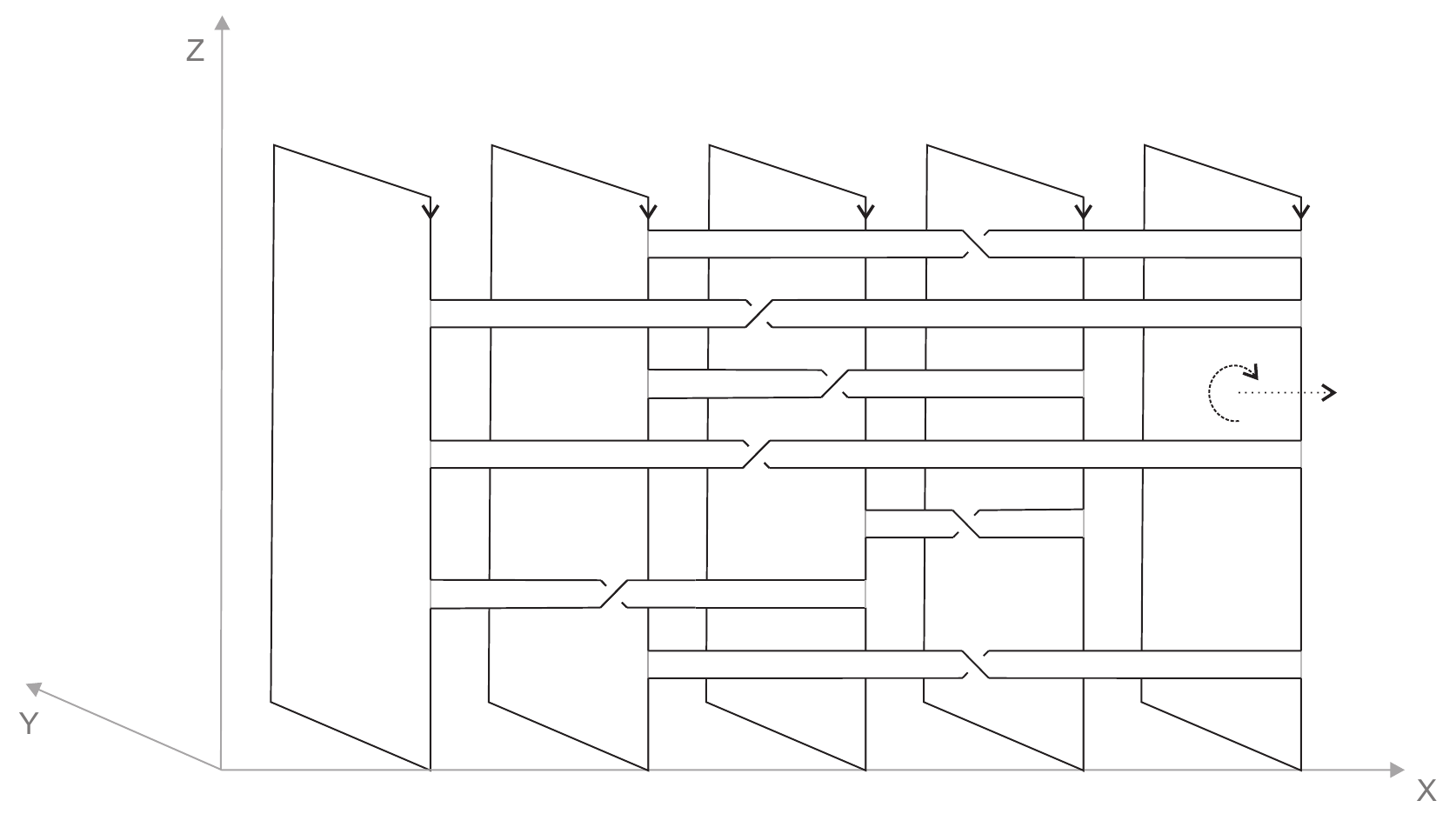}
\caption{\small{Braided surface $S = S(w)$, with $w = \sigma_{25}^{-1}\sigma_{15}\sigma_{24}\sigma_{15}\sigma_{34}^{-1}\sigma_{13}\sigma_{25}^{-1}$. In this example, $L(z_1) = 2$, $R(z_1) = 5$, $e(z_1) = -$, $L(z_6) = 1$, $R(z_6) = 3$ and $e(z_6) = +$.}}
\label{braidedsurface}
\end{figure}

We will say a Seifert surface is braided (see Figure \ref{braidedsurface}) if it has a handle-decomposition
$$
S= \left( \bigcup_{x \in  X} d_x \right) \, \, \bigcup \, \,  \left( \bigcup_{z \in Z } b_z \right)
$$
where:

\vspace{0.2cm}

\noindent $\bullet$ $X = \{1,2,...,n\}$ for a certain $n$, and if $x \in X$ then $d_x:= \{x\} \times [0,1] \times [0,k]$ is a $(1 \times k)$-rectangle parallel to the plane $YZ$ at a distance $x$ from it. Orient each $d_x$ such that its normal vector has the same orientation as the $X$-axis.

\vspace{0.2cm}

\noindent $\bullet$ $Z \subset \mathbb{R}$ is a finite set, and for each $z \in Z$, $b_z$ is a band joining two different discs, $d_{x_0}$ and $d_{x_1}$, with $x_0 < x_1$; let $P = (x,y,z') \in b_z$, then $x_0 \leq x \leq x_1$, $y\leq 0$ and $z' \in [z - \varepsilon, z + \varepsilon]$, $\varepsilon > 0$. Moreover, the intersection of $b_z$ with each plane $x = t \in [x_0, x_1]$ is a segment of length $2\varepsilon$ whose center is a point of height $z$.

\vspace{0.2cm}

\noindent $\bullet$ The attaching regions are those segments where $b_z$ intersects the plane $y = 0$; they are given by $\{x_0\} \times \{0\} \times [z - \varepsilon, z + \varepsilon]$ and $\{x_1\} \times \{0\} \times [z - \varepsilon, z + \varepsilon]$. When $t$ travels from $x_0$ to $x_1$, the segment $b_z \bigcap \{x = t\}$ makes a half twist, which can be positive or negative, as can be seen in Figure \ref{braidedsurface}.

\vspace{0.2cm}

We define maps $L,R: Z \longrightarrow X$ by the condition that a band $b_z$ joins the discs $d_{L(z)}$ and $d_{R(z)}$ with $L(z) < R(z)$. Also, $e: Z \longrightarrow \{+,-\}$ assigns to each $z \in Z$ the sign of the half twist of the band $b_z$.

\vspace{0.2cm}

The boundary of a braided surface $S$ (with a certain handle-decomposition) is an oriented link, closure of a braid represented by the BKL-word $w_S = \sigma_{L(z_1) R(z_1)}^{e(z_1)} \cdots \sigma_{L(z_b) R(z_b)}^{e(z_b)}$, where $Z = \{z_1, \ldots, z_b\}$, with $z_1 > z_2 > \dots > z_b$. Reciprocally, given a BKL-word $w$, there is a unique (up to isotopy) braided surface $S(w)$ such that $w_{S(w)} = w$. It is important to remark that different words representing the same braid may determine non isotopic braided surfaces; as an example consider the braided surfaces $S(\sigma_{12} \sigma_{12}^{-1})$ and $S(\mathbf{1})$. However, if $S=S(w)$ and $S'=S(w')$ are two ambient isotopic braided surfaces, then $w$ and $w'$ are BKL-words representing braids whose closures are equivalent links.

\begin{Def}
A surface $S \subset S^3$ is said to be BKL-positive (BKL-negative) if it is ambient isotopic to a braided surface $S' = S(w)$, with $w$ a positive (negative) BKL-word.
\end{Def}

The boundary of a BKL-positive surface is a BKL-positive link. Let us remark that BKL-positive surfaces are those classically named quasipositive surfaces, and BKL-positive links are known as strongly quasipositive links in the literature (see for example \cite{Rudolph1}, \cite{Rudolph2}).

\begin{Def}
A BKL-homogeneous braided surface is a braided surface $S(w)$ where $w$ is a homogeneous BKL-word. A surface is BKL-homogeneous if it is ambient isotopic to a BKL-homogeneous braided surface.
An Artin-homogeneous braided surface is a braided surface $S(w)$, with $w$  a homogeneous Artin-word. A surface is Artin-homogeneous if it is ambient isotopic to an Artin-homogeneous braided surface.
\end{Def}

Clearly, BKL-positive and BKL-negative surfaces are BKL-homogeneous.

\vspace{0.2cm}

Note that a link is BKL-homogeneous if and only if it is the boundary of a BKL-homogeneous surface. Although the boundary of an Artin-homogeneous surface is a homogeneous link, there are homogeneous links which are not boundary of an Artin-homogeneous surface, as we saw in Proposition \ref{homnobraidhom}.

\vspace{0.2cm}

\begin{figure}[t]
\centering
\includegraphics[width = 8.5cm]{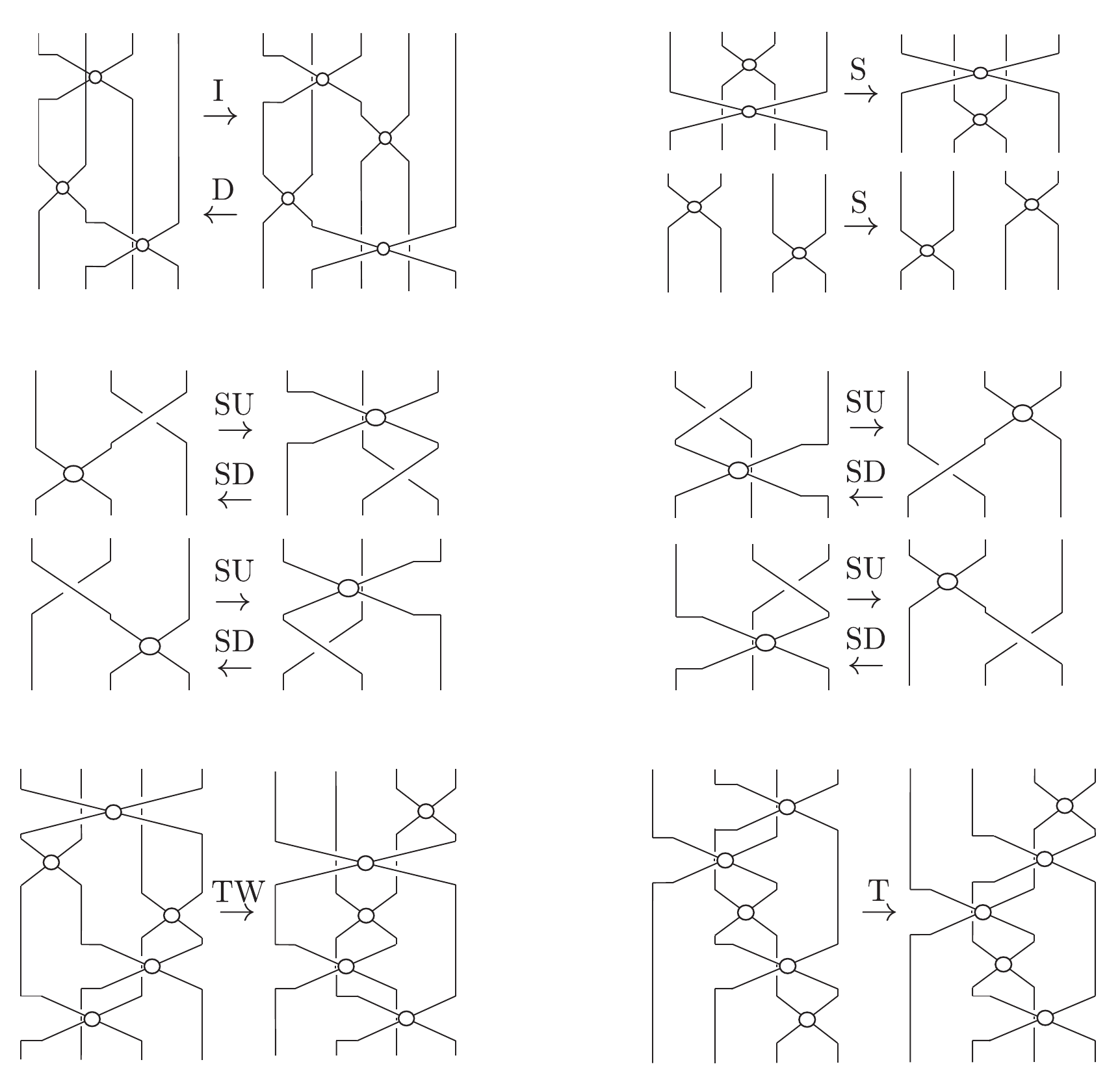}
\caption{\small{These movements are ambient isotopies between braided Seifert surfaces. A small circle represents a positive or negative crossing.}}
\label{movimientos}
\end{figure}

The following operations, defined in \cite{Rudolph1} and shown in Figure \ref{movimientos}, carry a braided surface $S=S(w)$ to an ambient isotopic one, $S'=S(w')$:

\vspace{0.2cm}

\noindent -\underline{Inflation} (I): given an strand $i$ and a sign $\varepsilon$, an inflation consists on replacing $w= uv$ with $w' = u' \sigma_{i i+1}^\varepsilon v'$, where $u'$ ($v'$) equals $u$ ($v$) after replacing each $\sigma_{jk}$ appearing in $u$ ($v$) by $\sigma_{f(j)f(k)}$, with $f(a) = a$ if $a \leq i$ and $f(a) = a+1$ if $a > i$. This operation inserts a disc and a band to $S$.  The opposite is called deflation (D). In the braid setting, these movements are natural generalizations of the stabilization/destabilization movements.

\vspace{0.2cm}

\noindent -\underline{Slip} (S): permute $\sigma_{ij}$ and $\sigma_{kl}$ if i and j do not separate k and l, or k and l do not separate i and j. In the surface, it exchanges the height of two consecutive unlinked bands. It is related to the first relation of the BKL presentation of $B_n$.

\vspace{0.2cm}

\noindent -\underline{Slide up} (SU): close to the second relation in the BKL presentation of $B_n$, when $i<j<k$ it replaces subwords $\sigma_{jk}^{+1}\sigma_{ij}^{\pm 1}$, $\sigma_{ij}^{-1}\sigma_{jk}^{\pm 1}$, $\sigma_{ij}^{+1}\sigma_{ik}^{\pm 1}$ or $\sigma_{jk}^{-1}\sigma_{ik}^{\pm 1}$ in $w$ with $\sigma_{ik}^{\pm1}\sigma_{jk}^{+1}$, $\sigma_{ik}^{\pm 1}\sigma_{ij}^{-1}$, $\sigma_{jk}^{\pm 1}\sigma_{ij}^{+1}$ or $\sigma_{ij}^{\pm 1}\sigma_{jk}^{-1}$ in $w'$, respectively. In the surface, it can be thought as sliding the lower band (in the pair of bands involved in the movement) over the other one, using one of its boundary arcs as a rail. The opposite is called slide down (SD).

\vspace{0.2cm}

\noindent -\underline{Twirl} (TW): it passes the leftmost string in the braid $\beta$, represented by $w$, to the rightmost position. The BKL-word $w'$ will be obtained from $w$ after replacing each $\sigma_{ij}$ with $\sigma_{i-1 j-1}$ if $i \neq 1$, or $\sigma_{j-1 n}$ if $i=1$, with $n$ the number of strands in $\beta$.

\vspace{0.2cm}

\noindent -\underline{Turn} (T): it sends the last BKL-generator appearing in $w$ to the first position; in the surface, the lowest band slides up behind the discs to the highest position. In the braid setting, this movement corresponds to a conjugation.

\section{Stallings plumbing and braided Stallings plumbing} \label{seccionplumbing}

Let $S \subset S^3$ be a Seifert surface; let $S^2 \subset S^3$ be a sphere separating $S^3$ into two non empty 3-balls, $B_1$ and $B_2$, such that $B_1 \bigcup B_2 = S^3$ and $B_1 \bigcap B_2 = S^2$. Let $N = S \bigcap S^2$, $S_1 = S \bigcap B_1$ and $S_2 = S \bigcap B_2$. If $S_i$ are Seifert surfaces and for some $n_i$, $i = 1,2$, $N \subset S_i$ is a $n_i$-patch in $S_i$, we say that $S^2$ deplumbs $S$ into two plumbands, and that $S$ is the Murasugi sum or Stallings plumbing of $S_1$ and $S_2$. We write $S = S_1 * S_2$.

\vspace{0.2cm}

If we try to define the operation $S = S_1 * S_2$ starting from the plumbands some additional information is needed: given two Seifert surfaces, $S_1$ and $S_2$, we need to specify the gluing $n_i$-patches $N_i \subset S_i$, $i = 1,2$, and an orientation-preserving homeomorphism $h: N_1 \longrightarrow N_2$, with $h (N_1 \bigcap \partial S_1) \bigcup (N_2 \bigcap \partial S_2) = \partial N_2$ (that is, the image under $h$ of the boundary arcs of $N_1$ in $S_1$ covers the interior arcs of $N_2$ in $S_2$). See \cite{Rudolph1} for more details.

\vspace{0.2cm}

When both $S_1 = S(w_1)$ and $S_2 = S(w_2)$ are braided surfaces, a special Stallings plumbing called braided Stallings plumbing can be performed: If $S_1$ and $S_2$ have handle decompositions as in Section \ref{seccionbraidedsurfaces}, the new surface $S = S(w) = S(w_1)*S(w_2)$ is defined as a Stallings plumbing using the first disc $d_1^2$ in $S(w_2)$ and the last disc $d_{n_1}^1$ in $S(w_1)$ as gluing discs.

\vspace{0.2cm}

The surface $S = S(w) = S_1 * S_2$ will be a braided Seifert surface with $\#_d (S) = \#_d(S_1) + \#_d(S_2) - 1$ and $\#_b (S) = \#_b(S_1) + \#_b(S_2)$; the order in the bands involving $d_{n_1}^1 = d_1^2$ will be given by a shuffling of the bands involving $d_{n_1}^1$ and $d_1^2$ in $S(w_1)$ and $S(w_2)$ respectively, keeping their internal order in the original surfaces.

\vspace{0.2cm}

\begin{figure}[t]
\centering
\includegraphics[width = 14cm]{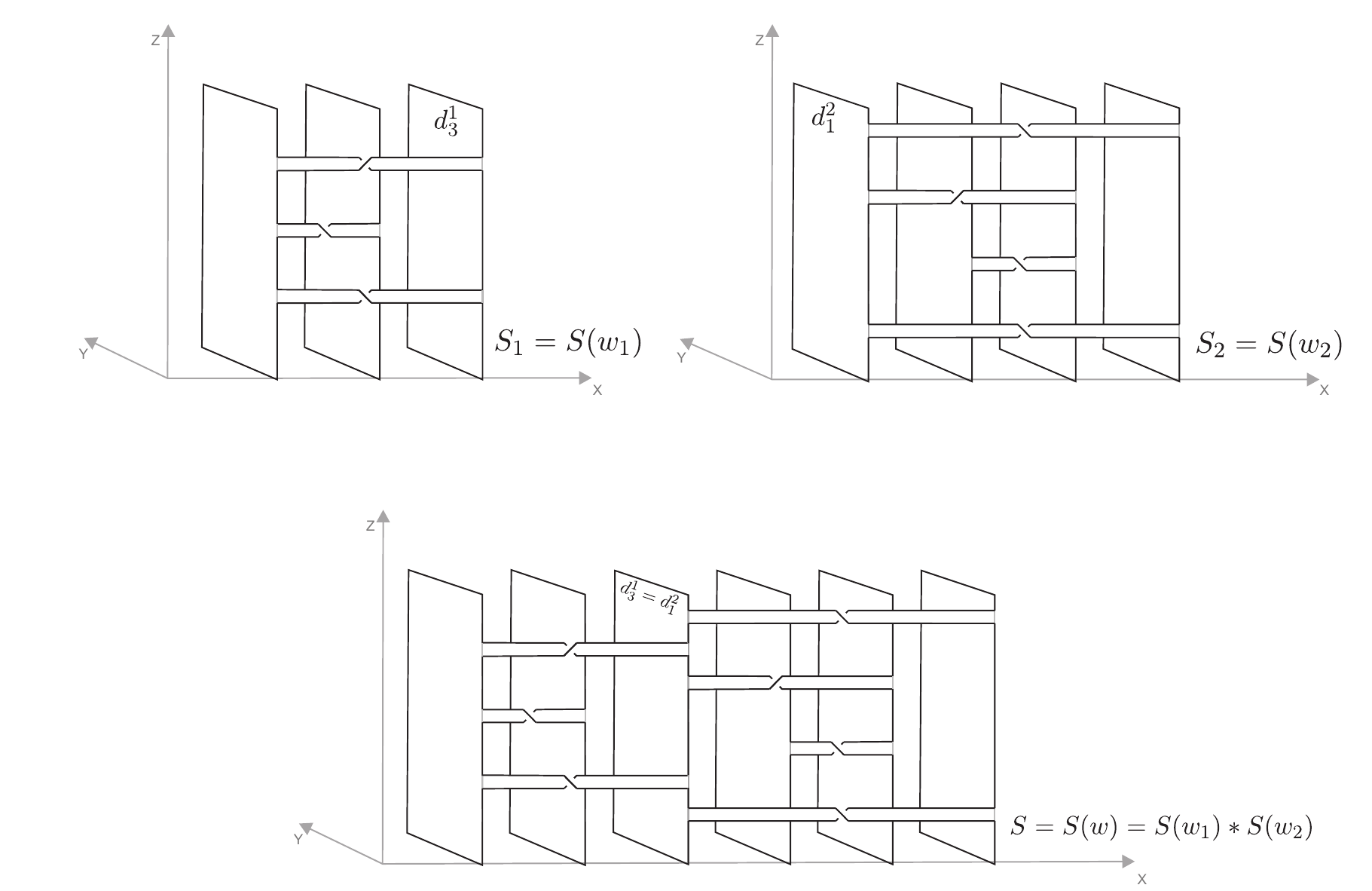}
\caption{\small{Surfaces $S_1= S(w_1)$ and $S_2 = S(w_2)$ where $w_1 = \sigma_{13}\sigma_{12}^{-1}\sigma_{13}^{-1}$ and $w_2= \sigma_{14}^{-1}\sigma_{13}\sigma_{23}^{-1}\sigma_{14}^{-1}$. After a shuffling of $w_2$ and $w_2' = \sigma_{36}^{-1}\sigma_{35}\sigma_{45}^{-1}\sigma_{36}^{-1}$, we obtain $w = \sigma_{36}^{-1}\sigma_{13}\sigma_{35}\sigma_{12}^{-1}\sigma_{45}^{-1}\sigma_{13}^{-1}\sigma_{36}^{-1}$, which represents the braid whose closure bounds $S = S(w)$, a braided Stallings plumbing of $S_1$ and $S_2$. }}
\end{figure}

Algebraically, $w$ can be thought as a resulting BKL-word when shuffling $w_1$ and $w_2'$, with $w_2'$ the BKL-word obtained after replacing each generator $\sigma_{ij}$ in $w_2$ with $\sigma_{i'j'}$, where $i' = i + n_1 - 1$ and $j' = j + n_1 - 1$, $n_1$ the number of strands in $\beta_1 = [w_1]$.

\section{BKL-homogeneity under Stallings plumbing}\label{new}

The following result was shown in \cite[Lemma 4.1.4]{Rudolph1} in the particular case of quasipositive surfaces (BKL-positive surfaces in this paper). We use the same techniques to extend the result to the case of BKL-homogeneous surfaces.

\begin{Lema} \label{reducefi}
Let $S = S(w)$ be a BKL-homogeneous braided surface and $\varphi \subset S$ a minimal braided n-star with $\delta_b(\varphi) > 0$. Then there exists a BKL-homogeneous braided surface $S' = S(w')$ and an ambient isotopy of $\mathbb{R}^3$ carrying $(S, \varphi)$ to $(S', \varphi')$, where $\varphi'$ is a minimal braided n-star with $\delta_b(\varphi') < \delta_b(\varphi)$. This construction can be done in such a way that $\#_d(S') = \#_d(S) + 1$ and $\#_b(S') = \#_b(S) + 1$.
\end{Lema}

\begin{proof}

Since $\delta_b(\varphi) > 0$, the n-star $\varphi$ has a long ray $\tau \not\subset S^d$. $\tau$ can be chosen so that $tail(\tau)$ is innermost (that is, the segment $\overline{coccyx(\tau) tip(\tau)}$ does not contain the endpoint $tip(\tau')$ of other ray $\tau' \subset \varphi$). Notice that there could be arcs of the star (not tails) in the region determined by $tail(\tau)$ and the segment  $\overline{coccyx(\tau) tip(\tau)}$.

Let $d_{x_0} = d_{tail(\tau)}$ and $b_{z_0} = b_{coccyx(\tau)}$. Let $z_{tip(\tau)}$ be the $z$-coordinate of $tip(\tau)$. Rotating the surface to put it upside-down if necessary, we can assume that $z_{tip(\tau)} > z_0$. By taking the mirror image if necessary (inverting the $x$-direction) we can assume that $e(b_{z_0}) = +$. And applying some twirls if necessary, we can assume that $x_0 = L(b_{z_0})$, that is, the disc attached to $d_{x_0}$ by $b_{z_0}$ is on its right.

\vspace{0.2cm}

\noindent
1) A picture of regions of $d_{x_0}$ and $d_{x_1}$, with $x_1:= R(b_{z_0})$, attached by $b_{z_0}$ is shown in Figure~\ref{dibulema} (the neighborhood of $tail(\tau)$ is included in the picture). Note that $x_1$ is not necessarily the successor of $x_0$ in $X$. Write $B^{\tau}:= \{ b_{z} \in S^b, \, \, z \in (z_0, z_{tip(\tau)})$\}, $B^{\tau}_{x_0}:= \{ b_{z} \in B^{\tau} \,$ such that either $L(b_z) = x_0$ or $R(b_z) = x_0 \} \subset B^{\tau}$. As $\varphi$ is minimal, $\tau$ is not loose, so the set $B^{\tau}_{x_0}$ is not empty; note that $B^{\tau}_{x_0}$ may contain bands attaching $d_{x_0}$ and $d_{x_1}$, but as $S$ is a BKL-homogeneous braided Seifert surface, their sign would be equal to $e(b_{z_0}) = +$. Write $z_1:= \displaystyle\max_{b_z \in B^{\tau}}{z}$, that is, $b_{z_1}$ is the highest band between $coccyx(\tau)$ and $tip(\tau)$ (Figure $\ref{dibulema}_A$).

\vspace{0.2cm}

\noindent
2) Perform an inflation of sign $e(b_{z_0}) = +$ introducing a new band, $b_{\widetilde{z}}$, and a new disc, $d'_{x_0^+}$, so that $\widetilde{z} \in (z_1, z_{tip(\tau)})$ and of course ${x_0^+}$ is the successor of $x_0$ in $X'= X \bigcup \{n+1\}$ (Figure~{$\ref{dibulema}_B$}); each $d_{x_j}$ with $x_j > x_0$ will be pushed to the right, thus it becomes $d'_{x_j + 1}$, and $d'_{x_0^+} = d'_{x_0 + 1}$. As $b_{\widetilde{z}}$ is the only band attaching $d'_{x_0}$ and $d'_{x_0 + 1}$, the resulting surface is braided BKL-homogeneous.

\vspace{0.2cm}

\begin{figure}
\centering
\includegraphics[width = 11cm]{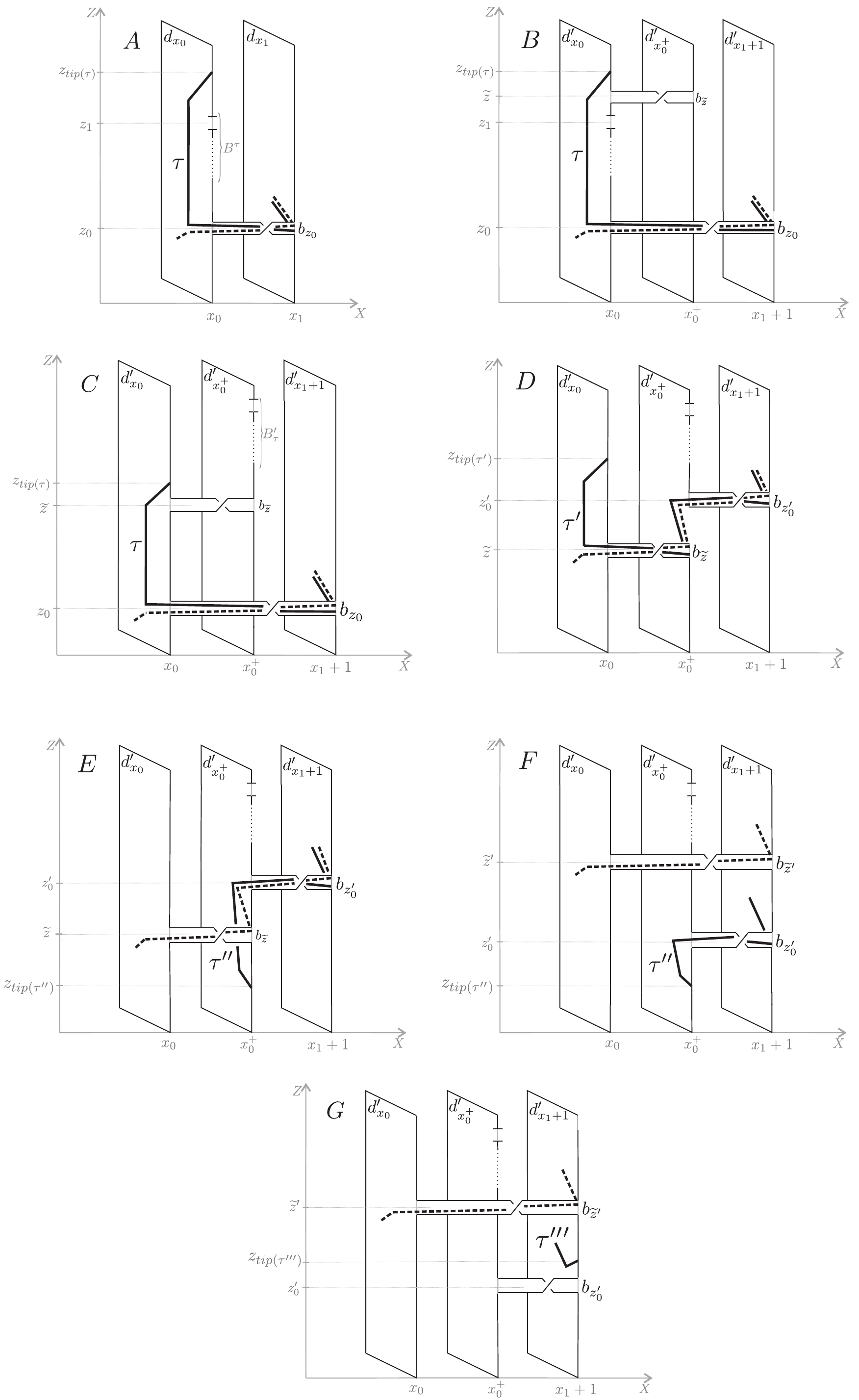}
\caption{\small{This figure illustrates the steps followed in proof of Lemma \ref{reducefi}.
}}
\label{dibulema}
\end{figure}

\noindent
3) As ${x_0 + 1}$ is the successor of $x_0$, the set $B^{\tau}$ can be slipped or slid up keeping the original order (concretely a slide up acts on those bands in  $B^{\tau}_{x_0}$ while a slip up is performed on the others). The slid up bands have just changed one of their attaching regions from $d_{x_0}$ to $d'_{x_0 + 1}$, as can be seen in Figure $\ref{dibulema}_C$, so the surface is still braided BKL-homogeneous.
\vspace{0.2cm}

\vspace{0.2cm}

\noindent
4) Sliding up $b_{z_0}$ over $b_{\widetilde{z}}$ (giving a new band $b_{z_0'}$) transforms $\tau$ into $\tau'$, which is loose (Figure $\ref{dibulema}_D$), and $\varphi$ into $\varphi'$. Apart from $\tau$, other rays of $\varphi$ can cross $b_{z_0}$. In fact, if there are exactly $k+1$ arcs of $\varphi$ contained in $b_{z_0}$ then $\delta_b(\varphi') = \delta_b(\varphi)+k-1$.

\vspace{0.2cm}

\noindent
5) Since $\tau$ is loose, $tip(\tau')$ can be pulled along the border of the surface as shown in Figure $\ref{dibulema}_E$, arising the new ray $\tau''$ in the new star $\varphi''$. Clearly $\delta_b(\tau'') = \delta_b(\tau') - 1$ and $\delta_b(\varphi'') = \delta_b(\varphi') - 1 = \delta_b(\varphi) + k - 1$. Notice that every arc originally contained in the region determined by $tail(\tau)$ and the segment joining $tip(\tau)$ and $coccyx(\tau)$ slides up to the new disc, without interfering with the other parts of the star, and without increasing the $\delta_b$ value of its ray. Comparing the original surface with the actual one, we have just transferred the attaching regions of the bands $b_z \in B^{\tau}_{x_0}$ in $\partial(d_{x_0})$ to $\partial(d'_{x_0 + 1})$, keeping the original order and signs, so at this point the surface is braided BKL-homogeneous.

\vspace{0.2cm}

\noindent
6) Slide up $b_{\widetilde{z}}$ over the band joining $d'_{x_0 + 1}$ and $d'_{x_1+1}$, that is, over $b_{z_0'}$, as shown in the picture, Figure $\ref{dibulema}_F$.

\vspace{0.2cm}

\noindent
7) The previous movement makes $\tau''$ loose again, so its final point, $tip(\tau'')$, can be pulled along the border of the surface as before (Figure $\ref{dibulema}_G$), arising the new ray $\tau'''$ in the new star $\varphi'''$. Clearly $\delta_b(\tau''') = \delta_b(\tau'') - 1 = \delta_b(\tau) - 1$; as other arcs crossing $b_{z_0}$ in the original situation keep their original $\delta_b$ value, we have $\delta_b(\varphi''')<\delta_b(\varphi)$.
\end{proof}

\begin{Cor} \label{estrellaendisco}
If $\varphi$ is an n-star on a BKL-homogeneous surface $S$, then there exists an ambient isotopy of $\mathbb{R}^3$ carrying $(S, \varphi)$ to $(S', \varphi')$, with $S'= S(w')$ a BKL-homogeneous braided surface and $\varphi' \subset S'^d$ a minimal braided n-star (that is, there is a disc of the standard handle decomposition of $S'$ containing $\varphi'$).
\end{Cor}

\begin{proof}
By definition of BKL-homogeneous surface, there exists an ambient isotopy $I_t$ of $\mathbb{R}^3$ such that $I_1(S) = \overline{S} = S(\overline{w})$, with $\overline{w}$ a homogeneous BKL-word. Consider $I_1 (\varphi) = \overline{\varphi}$, which is an n-star on the surface $\overline{S}$, and reduce it till it is minimal (of course, this isotopy is easily extended to an ambient isotopy of the whole $\mathbb{R}^3$). If $\delta_b(\overline{\varphi}) = 0$ the n-star does not cross any band, so  $\overline{\varphi} \subset S'^d$ and $\varphi' = \overline{\varphi}$ and $S' = \overline{S}$. Otherwise, apply repeatedly Lemma \ref{reducefi}, which preserves the BKL-homogeneous character of the surface, until the image of $\varphi$ does not cross any band.
\end{proof}

\begin{Prop} \label{pegadohomogeneas}
Let $S_1$, $S_2$ be two BKL-homogeneous surfaces and let $S = S_1 * S_2$ be their Stallings plumbing; then there exist three BKL-homogeneous braided surfaces $S'$, $S_1'$ and $S_2'$ ambient isotopic to $S$, $S_1$ and $S_2$ respectively, such that $S' = S_1' * S_2'$ is the braided Stallings plumbing of $S_1'$ and $S_2'$.
\end{Prop}

\begin{proof}
Let $\varphi_1 \subset S_1$, $\varphi_2 \subset S_2$ be the $n_1$-star and $n_2$-star used in the original plumbing, respectively. Applying Corollary \ref{estrellaendisco} we can assume that there exist two isotopies $I_t$ and $J_t$ such that $I_1(S_1)= S_1'$ with $S_1'$ a BKL-homogeneous braided surface with $I_1(\varphi_1) = \varphi_1'$ contained in a disc of $S_1'$, and the analogous statement for $J_t$ and $S_2$. Applying twirls to a surface does not affect its BKL-homogeneity, so we can consider $\varphi_1'$ contained in the rightmost disc in $S_1'$ and $\varphi_2'$ in the leftmost disc in $S_2'$. The order in the bands in the original plumbing was given by an orientation-preserving homeomorphism $h: N_1 \subset S_1 \longrightarrow N_2 \subset S_2$; so, up to turns of the original surfaces, the order in the bands in the plumbing $S_1' * S_2'$ will be given by $J_t \circ h \circ I_t^{-1}$. By construction, the surface $S_1' * S_2'$ is ambient isotopic to $S$.
\end{proof}

\begin{Prop} \label{suficiente}
If $S$ is a BKL-homogeneous surface deplumbed into surfaces $S_1$ and $S_2$, then $S_1$ and $S_2$ are BKL-homogeneous.
\end{Prop}

\begin{proof}
Since $S = S_1 * S_2$, there exists a sphere $S^2$ separating $S^3$ as in Section \ref{seccionplumbing}, with $N = S_1 \bigcap S_2 \subset S^2$ being an  $n_i$-patch in $S_i$, denoted by $N_i$, regular neighborhood of a certain $n_i$-star $\varphi_i$ in $S_i$, $i=1,2$; the ``gluing order'' is given by $h: N_1 \longrightarrow N_2$. We can take $\varphi_1$ and $\varphi_2$ so that $\psi = \varphi_1 \bigcup \varphi_2$ is an $(n_1 + n_2)$-star on $N$. As $S$ is a BKL-homogeneous surface, by Corollary \ref{estrellaendisco} there exists an isotopy, $I_t$, carrying $S$ to $S' = S(w')$, with $w'$ a homogeneous BKL-word and $I_1(\psi)$ contained in a disc, $\widehat{d}$. As $S^2$ separates $S_1$ and $S_2$ in the original situation, $I_1(S^2)$ separates $I_1(S_1) = S_1'$ and $I_1(S_2) = S_2'$, that is, there are no bands attaching discs from $S_1'$ to discs from $S_2'$ other than $\widehat{d}$. Note that neither discs in $S_1$ nor discs in $S_2$ are necessarily consecutive. As $S'$ is a BKL-homogeneous braided surface, $S_1'$ and $S_2'$ are.
\end{proof}

\vspace{0.2cm}

The main theorem follows from Propositions \ref{pegadohomogeneas} and \ref{suficiente}:

\vspace{0.2cm}

\noindent \textbf{Theorem \ref{Teorema}.}
\emph{Let $S$ be a Stallings plumbing of Seifert surfaces $S_1$ and $S_2$; then $S$ is a BKL-homogeneous surface if and only if both $S_1$ and $S_2$ are BKL-homogeneous surfaces.}

\begin{Cor}\label{corolariopseudoalt}
Any generalized flat surface spanning a pseudoalternating link is a BKL-homogeneous surface. In particular, any pseudoalternating link is a BKL-homogeneous link.
\end{Cor}

\begin{proof}
By definition, a generalized flat surface is constructed by performing a finite number of Stallings plumbings of primitive flat surfaces. Then, by Theorem \ref{Teorema}, it is enough to show that primitive flat surfaces are BKL-homogeneous. Such a surface is the projection surface of either a positive or a negative diagram (with no nested Seifert circles). In \cite{Rudolph2} it is proved that the projection surface of a positive diagram is a BKL-positive surface (called quasipositive in Rudolph's paper); with the same argument applied to the mirror image, it follows that projection surfaces constructed from negative diagrams are BKL-negative. And both BKL-positive and BKL-negative surfaces are obvious examples of BKL-homogeneous surfaces.
\end{proof}

\begin{Cor} \label{corolariohom1}
Let $D$ be an oriented homogeneous diagram of an oriented homogeneous link $L$. Then its projection surface is a BKL-homogeneous surface. In particular, any homogeneous link is a BKL-homogeneous link.
\end{Cor}

\begin{proof}
As $L$ is a homogeneous link, it is pseudoalternating. In fact, the projection surface constructed from $D$, $S_D$, is a generalized flat surface spanning $L$; thus $S_D$ is BKL-homogeneous and $L$ is a BKL-homogeneous link.
\end{proof}

We just want to mention an example of homogeneous link which has a minimal non-homogeneous diagram in the sense of minimal crossing number. This question was posed by Peter Cromwell in \cite{CromwellHom}, and it was motivated by the fact that a non-alternating diagram of a prime alternating link cannot have minimal crossing number \cite{Murasuginonaltnonmim}. The example is the Perko’s knot given by the equivalent diagrams $10_{161} \equiv 10_{162}$, as the diagram $10_{161}$ is minimal and non-homogeneous, and the diagram $10_{162}$ is positive hence homogeneous.

\vspace{1.1cm}

\noindent \textbf{Acknowledgements} \,
I want to thank Pedro M. G. Manchón and Juan González-Meneses for introducing me in the study of braided surfaces. I am also grateful for their numerous valuable comments, their suggestions and corrections on preliminary versions of this paper. I would also like to thank Józef H. Przytycki for our conversations, which were very useful and helped me to clarify ideas. \\

\bibliographystyle{plain}
\bibliography{Bibliograf}

\begin{thebibliography}{10}

\bibitem{Artin1}
E.~Artin.
\newblock Theorie der {Z}öpfe.
\newblock {\em Abh. Math. Sem. Hamburg}, 4:47--72, 1925.

\bibitem{Artin2}
E.~Artin.
\newblock Theory of {B}raids.
\newblock {\em Ann. of Math.}, 48:101--126, January, 1947.

\bibitem{Baader}
S.~Baader.
\newblock Quasipositivity and homogeneity.
\newblock {\em Math. Proc. Camb. Phil. Soc.}, 139:287--290, 2005.

\bibitem{knotatlas}
Dror Bar-Natan, Scott Morrison, and et~al.
\newblock The {K}not {A}tlas.
\newblock http://katlas.org.

\bibitem{BirmanKoLee}
J.~Birman, K.H. Ko, and S.J. Lee.
\newblock A new approach to the {W}ord and {C}onjugacy {P}roblems in the
  {B}raid {G}roups.
\newblock {\em Advances in Math.}, 139:322--353, 1998.

\bibitem{LibroCromwell}
P.~Cromwell.
\newblock {\em Knots and Links}.
\newblock Cambridge University Press, 2004.

\bibitem{CromwellHom}
P.R. Cromwell.
\newblock Homogeneous links.
\newblock {\em J. Lond. Math. Soc.}, 2(39):535--552, 1989.

\bibitem{LibroKauffman}
L.H. Kauffman.
\newblock {\em Formal Knot Theory}.
\newblock Princeton University Press, 1983.

\bibitem{Pseudoalternantes}
E.J. Mayland and K.~Murasugi.
\newblock On a structural property of the groups of alternating links.
\newblock {\em Canad. J. Math.}, XXVIII(3):568--588, 1976.

\bibitem{Murasuginonaltnonmim}
K.~Murasugi.
\newblock Jones polynomials and classical conjectures in knot theory.
\newblock {\em Topology}, 26:187--194, 1987.

\bibitem{Rudolph1}
Lee Rudolph.
\newblock Quasipositive plumbing (constructions of quasipositive knots and
  links, {V}).
\newblock {\em Proc. Amer. Math. Soc.}, 126(1):257--267, 1998.

\bibitem{Rudolph2}
Lee Rudolph.
\newblock Positive links are strongly quasipositive.
\newblock {\em Geometry and Topology Monographs: Proceedings of the Kirbyfest},
  2:555--562, 1999.

\bibitem{Stallings}
J.~Stallings.
\newblock Constructions of fibred knots and links.
\newblock {\em Algebraic and geometric topology (Proc. Sympos. Pure Math.,
  XXXII, Amer. Math. Soc.)}, 2:55--60, 1978.

\end{thebibliography}

\end{document}